\documentclass[10pt]{amsart}
\usepackage{amsmath,amssymb,amsfonts}
\usepackage{amscd}
\usepackage{graphics}
\usepackage{MnSymbol}

\newcommand\cA{{\mathcal A}}
\newcommand\cB{{\mathcal B}}

\newcommand\cD{{\mathcal D}}

\newcommand\cF{{\mathcal F}}

\newcommand\cL{{\mathcal L}}

\newcommand\cU{{\mathcal U}}

\newcommand\BN{{\mathbb{N}}}

\newcommand\BQ{{\mathbb{Q}}}
\newcommand\BR{{\mathbb{R}}}

\newcommand\BZ{{\mathbb{Z}}}

\newcommand\ba{{\bold{a}}}

\newcommand\bd{{\bold{d}}}
\newcommand\be{{\bold{e}}}

\newcommand\br{{\bold{r}}}

\newcommand\bx{{\bold{x}}}
\newcommand\by{{\bold{y}}}

\newtheorem{thm}{Theorem}

\newtheorem{cor}[thm]{Corollary}

\newtheorem{lemma}[thm]{Lemma}
\newtheorem{definition}[thm]{Definition}
\newtheorem{prop}[thm]{Proposition}
\theoremstyle{remark}
\newtheorem{example}{Example}

\title{Parry's topological transitivity and $f$-expansions}
\date{May 18, 2014}
\author{E. Arthur Robinson, Jr.}
\address{Department of Mathematics\\ George Washington University\\ 2115 G St. NW\\ 
Washington, DC 20052}
\email{robinson@gwu.edu}
\begin{document}

\begin{abstract}
In his 1964 paper \cite{parry2} on $f$-expansions, Parry studied piecewise-continuous, 
piecewise-monotonic  maps $F$ of the interval $[0,1)$,  
and introduced a notion of topological transitivity different from any of the modern definitions. 
This notion, which we call {\em Parry topological transitivity}, 
(PTT) is that the {\em backward orbit} 
$O^-(x)=\{y:x=F^ny\text{\ for\ some\ }n\ge 0\}$ of some $x\in[0,1)$ is dense. 
We take {\em topological transitivity} (TT) to mean that some $x$ has a dense {\em forward} 
orbit.
Parry's application to $f$-expansions is that PTT implies the partition of $[0,1)$ into the ``fibers'' 
of $F$ is a generating partition (i.e., $f$-expansions are ``valid''). We prove the same result for 
TT, and use this to show that
for interval maps $F$,   
TT implies PTT. A separate proof 
is provided for continuous maps $F$ of compact metric spaces. The converse is false.
\end{abstract}

\maketitle

\section{Introduction} 

The concept of {\em topological transitivity} plays an important role in dynamical systems 
theory. 
Let  $F:X\to X$ be a surjective map on a  topological space $X$.
The definition of topological transitivity (TT) that  we will adopt in this paper
is that for some $x\in X$
the {\em forward orbit}
\(
O^+(x)=\{F^nx:n\ge 0\}
\) 
is dense in $X$. 
Another definition, sometimes called 
{\em regional topological transitivity} 
(RTT), is that for any two non-empty open sets $U,V\subseteq X$, there exists $n>0$ so that  
\(
U\cap F^n V\not=\emptyset
\), or 
equivalently (see \cite{akin}),
\(
F^{-n}U\cap V\not=\emptyset
 \).
The equivalence of TT and RTT for continuous maps  $F$ of perfect compact metric spaces 
$X$ is well known 
(see Proposition~\ref{equiv} below).
Several papers (see for example \cite{kolyada} or \cite{akin}) 
discuss these, and other, definitions of topological transitivity 
for continuous maps $F$, 
and give conditions under which various definitions are equivalent.

However, one often wants to apply the concept of 
topological transitivity in situations with
less ideal hypotheses.
One benefit of the definition TT 
is that it makes sense even  
when $F$ is not continuous. In this paper, we will 
mostly be interested piecewise monotonic, piecewise continuous maps $F$ on
the unit interval.
In 1964, Parry \cite{parry2} gave a different 
definition of topological transitivity in this situation,    
which we refer to here as {\em Parry topological transitivity} (PTT).
It says that for some $x\in X$ the {\em backward orbit} 
\[
O^-(x)=\{F^{-n}(x):n\ge 0\}=\{y:x=F^n(y) \text{ for some }
n\ge 0\}
\]
is dense. 

As we 
will see, PTT generally does not 
imply TT, but in many situations, TT does imply PTT.
It is not  hard  to obtain such results under ``nice'' hypotheses, like for 
subshifts (see Corollary~\ref{nine}) or for continuous maps  $F$ of 
perfect compact metric spaces (see Theorem~\ref{ttiptt}). 
In this case, we show that TT implies that $O^-(x)$ is dense for a dense $G_\delta$ set of 
$x\in X$.
Recently, it was shown in \cite{anima} that  
continuous TT maps $F:\BR\to\BR$ satisfy PTT (in fact \cite{anima} proves more: 
if such an $F$ is TT then $O^-(x)$ is dense for all but possibly two points $x\in\BR$).

Our main goal in this paper is 
to understand 
the situation in the case studied by Parry \cite{parry2}, namely, 
for piecewise-continuous, 
piecewise-monotonic maps $F$ of the interval.
We call a surjective map $F:[0,1)\rightarrow [0,1)$ a {\em piecewise interval map} 
if there is a partition of (almost all of) $[0,1)$ into a finite or countable 
$\xi$ of disjoint intervals, indexed by a ``digit set'' $\cD$,
such that each 
$F|_{\Delta(d)}$ is continuous and strictly monotonic see Section~\ref{pim} for details).
In his paper \cite{parry2}, Parry considered piecewise interval maps  in the context of $f$-
expansions, 
as defined by R\`enyi \cite{renyi}, Bissinger \cite{bissinger} and Everett \cite{everett} in the 
1940's and 1950's. 
Unknown to these authors, the same idea  
had actually been 
studied earlier in 1929 by Kakeya \cite{Kakeya}. 

The idea of $f$-expansions (the term is due to R\`enyi, \cite{renyi})
is to use piecewise 
interval maps  $F$ to obtain what we call
the $F$-representation $\br(x)\in
\cD^{\BN}$ of  $x\in[0,1)$
by recording the sequence $\br(x)=.d_1d_2d_3\dots$
of $\xi$-intervals visited by the $F$-iterates of $x$ (see Section~\ref{pim}).  
The goal is to find conditions 
on $F$ so that almost every $x$ has a unique $F$-representation
(``valid'' in Parry's terminology \cite{parry2}). One also studies
an algorithm (see Section~\ref{fexpg}) to recover $x$  
from $\br(x)$. 
In particular, under appropriate conditions the 
``$f$-expansion''
\(f(d_1+f(d_2+f(d_3+\dots)))\) converges to $x$, where
$f:\BR\to[0,1]$ is a function satisfying $F(x)=f^{-1}(x)\text{\ mod\ }1$.

There are, of course, two examples of $F$-representations and $f$-expansions that are 
especially well known. 
Binary representations/expansions of real numbers 
correspond to 
$F(x)=2x\text{ mod }1$. with 
$f(x)=x/2$,\footnote{Decimal representation and decimal expansions correspond to replacing 
the ``base'' 2 with base 10.}
Continued fraction representations
correspond to $F(x)=1/x\text{ mod }1$. 
In each of these cases, $F$ satisfies both TT and PTT.

In \cite{parry2} Parry proved  that PTT implies $F$-representations are valid (and 
also, with some additional hypotheses, that valid $F$-expansions implies 
PTT). In this paper, 
we prove a slightly strengthened version of Parry's first result, as well as a 
``modern'' version of Parry's result, 
which says that TT implies $F$-representations are valid.
One benefit is that TT is often very easy to verify. For example, 
(see Proposition~\ref{ergodic}), any 
$F$ that is ergodic for an invariant measure $\mu$  
equivalent to Lebesgue measure will satisfy TT. 
In the end, we show 
(Theorem~\ref{mainresult})
that TT implies PTT for piecewise interval maps $F$.

\section{Topological transitivity} \label{topt}

In this section we consider various notions of topological transitivity for continuous maps $F$. 
We begin with two standard results mentioned in the introduction, which we
prove for the sake of completeness. 

\begin{prop}\label{equiv}
Suppose $F: X\rightarrow X$ is a continuous map on a compact metric space. If $F$ satisfies RTT then  
the set
$X_0=\{x:O^+(x)\text{\ is\ dense\ in\ } X\}$ contains a dense $G_\delta$. In particular, RTT implies TT. If, in addition, $X$ is perfect, then TT implies RTT.
\end{prop}

\begin{proof}
Assume $U \cap F^{-n} V\not=\emptyset$ for some $n\ge 0$. Then for any $V$ open, $\cup_{n\ge 0}F^{-n} V$ is dense open, since it meets any open set 
$U$. 
Thus,
whever $V_k$ is countable basis for $X$,
the Baire Category Theorem implies 
 $X_0=\bigcap_{k\ge 0}\cup_{n\ge 0}F^{-n} V_k$ is 
dense $\text{G}_\delta$.
Clearly $x\in X_0$ implies that for any $k$, there exists $n\ge 0$ so that $F^n(x)\in V_k$. Thus 
$O^+(x)$ is dense. 

Now suppose $O^+(x)$ is dense and let $U$ and $V$ be open.
There exist $n,m\ge 0$ with $F^nx\in U$ and $F^mx\in V$.  
Since $X$ 
is perfect, Lemma~\ref{infinite} (below) shows
we may assume $m>n$. 
It follows that $F^n x\in U\cap F^{-m+n}V\not=\emptyset$.  
\end{proof}

\begin{lemma}\label{infinite}
Let $X$ be a perfect (no isolated points) metric space (not necessarily compact), and suppose $F:X\to X$ is continuous. If $O^+(x)$ is dense, and $V\subseteq X$ is open, then $\{k\in\BN_0:F^k(x)\in V\}$ is infinite. 
\end{lemma}

\begin{proof}
Let $V_n=V\backslash\{F^k(x): x=0,1,\dots,n-1\}$. Since $X$ is perfect  metric, $V_n$ is nonempty and open, so $O^+(x)\cap V_n\not=\emptyset$. 
It follows that for any $n$ there is $k\ge n$ so that $F^k(x)\in V$.  
\end{proof}

\begin{prop}\label{invertible}
If $F:X\rightarrow X$ is a homeomorphism of a perfect compact metric space, then TT is equivalent 
to TTT.
In fact, $F$ and $F^{-1}$ both  satisfy TT, and 
there exists $X_0\subseteq X$, dense $G_\delta$ so that $O^-(x)$ and $O^+(x)$ are both dense for $x\in X_0$.
\end{prop}

\begin{proof}
Clearly, if $O^+(x)$ is dense then $O(x)$ is dense for all $x$ in a dense $G_\delta$. 
On the other hand, if $O(x)$ is dense then either 
$O^+(x)$ is dense (for all $x$ in a dense $G_\delta$), and 
thus $F$ is TT, or $O^-(x)$ is dense (for all $x$ in a dense $G_\delta$) and $F^{-1}$ is TT. 
In the latter case, $U \cap F^{-n} V\not=\emptyset$ for some $n\ge 0$, which 
shows $F$ is also TT. 

\end{proof}

\subsection{The relation between TT and PTT}

We begin with an example that shows  Parry topological transitivity (PTT) does not imply 
topological transitivity (TT).

\begin{example}\label{first}
Define a surjective map on a compact 
metric space by $F:[0,1]\to[0,1]$ by\footnote{My thanks to Ethan Akin for suggesting this 
example.}  
\[
F(x)=
\begin{cases}
-2x+1/2&\text{ if }x\in[0,1/4),\\
\phantom{-}2x-1/2&\text{ if }x\in[1/4,3/4),\text{\ and}\\
-2x+5/2&\text{ if }x\in[3/4,1].
\end{cases}
\]
Note that $O^-(1/2)$ is dense, whereas 
$F^n(1/8,3/8)\cap(5/8,7/8)=\emptyset$ for all $n\ge 0$. 
Note also that just a single point has $O^{-}(x)$, and not a dense $G_\delta$ set of points. 

\end{example}

In the other direction, we have the following:

\begin{thm}\label{ttiptt}
Suppose $F$ 
is a continuous map on a perfect compact metric space $X$ that satisfies 
TT. Then $F$ satisfies PTT, and moreover, the set 
$X_0=\{x:O^-(x)\text{\ is\ dense}\}$
contains a dense $G_\delta$.
\end{thm}

\begin{proof}
Define 
\[
\widetilde X=\{\tilde x=(x_1,x_2,x_3\dots)\in X^{\BN_0}:x_n=F(x_{n+1})\}.
\] 
This is a compact metric space with the 
topology 
induced by product topology, with metric
\[\tilde d(\tilde x,\tilde y)=\sum_{n\ge 1}d_X(x_n,y_n)/2^n,\]
where $d$ is the metric on $x$. 
Define ${\widetilde F}:{\widetilde X}\to{\widetilde X}$ by
\[
{\widetilde F}(x_1,x_2,x_3\dots)=(F(x_1), x_1,x_2,x_3\dots),
\]
and note that ${\widetilde F}^{-1}(x_1,x_2,x_3\dots)=(x_2,x_3,x_4\dots)$, so
${\widetilde F}$ is a homeomorphism. 
The map $\pi_k:\widetilde X\to X$, defined 
$\pi_k(x_1,x_2,x_3,\dots)=x_k$, 
is surjective and open. It follows that 
$\widetilde X$ is perfect since $X$ is perfect.
The sets $\widetilde U_k=\pi_k^{-1}(U)$, for  
$k=1,2,3,\dots$ and $U\subseteq X$ open, form a sub-base 
for the topology on $\widetilde X$. 
Given  a countable base  $\cU$ for  $X$, 
let $\widetilde\cU$ consist of all nonempty sets of the form 
\begin{equation}\label{base}
\widetilde U=\pi_1^{-1} (U_1)\cap \pi_{2}^{-1} (U_2)\cap\dots\cap \pi^{-1}_\ell(U_\ell)\subseteq\widetilde X,
\end{equation}
for some $\ell\ge 1$ and $U_1,U_2,\dots, U_\ell\in\cU$.
Then $\widetilde \cU$ is a countable base for $\widetilde X$. 

For $\widetilde U\in\cU$, let 
\begin{equation}\label{yous}
U= 
F^{-\ell+1}U_1
\cap
F^{-\ell+2}U_2
\cap\dots\cap 
U_\ell
\subseteq X,
\end{equation}
and assume $U$ is nonempty. 
Since $F$ satisfies TT, the set 
\[
X_U=\bigcup_{n\ge 0} F^{-n}U=\{x\in X:O^+(x)\cap U\not=\emptyset\}
\] 
is dense open, so
$X_0=\bigcap_{U\in\cU} X_U$ is dense $G_\delta$.
We {\bf claim} that $O^+(\tilde x)$ is dense in $\widetilde X$
for any $\tilde x\in\pi_1^{-1}(X_0)$, and thus 
$\widetilde F$ satisfies TT. 

To prove the claim, fix $x\in X_0$ and let 
$\tilde x=(x,x_2,x_3,\dots)\in\pi_1^{-1}(X_0)$.  
Note that for any $k\ge 1$,
\[
{\tilde F}^k(\tilde x)=(F^k(x),F^{k-1}(x),\dots,x,x_2,\dots).
\] 
Given
${\widetilde U}\in {\widetilde \cU}$, let $U$ be as in (\ref{yous}).
Since $O^+(x)$ is dense and $X$ is perfect, we can 
choose $n\ge \ell-1$ so that $F^{n-\ell+1} x\in U$.  
Then 
\[
F^n (x)\in U_1,\,F^{n+1}(x)\in U_2,\ \dots,
\,F^{n+\ell-1}(x)\in U_\ell, 
\]
and it follows from (\ref{base}) that  ${\widetilde F}^n(\tilde x)\in \widetilde U$.
Since $\widetilde U\in\widetilde\cU$ was arbitrary, $O^+(\tilde x)$ is dense,
proving the claim.

Now, since $F$ satisfies TT and $\widetilde F$ is a homeomorphism of a perfect metric space $\widetilde X$, it follows from Proposition~\ref{invertible} that $\widetilde F$ satisfies  TTT.  
Thus $O^-({\tilde x})$ dense for $\tilde x\in \widetilde X_0\subseteq \widetilde X$, 
where $\widetilde X_0$ contains a dense $G_\delta$.
Since $\pi_1$
is surjective and open, $X_0=\pi_1(\widetilde X_0)$ contains a 
dense $G_\delta$, and for $\tilde x\in \widetilde X_0$, 
$\pi_1(O^-({\widetilde x}))$ is dense in $X$.
But $\pi_1(O^-({\widetilde x}))\subseteq O^-(\pi_1({\widetilde x}))=O^-(x)$, so $F$ satisfies PTT. 
\end{proof}

For $F:X\to X$, we call a set $B^-\subseteq X$ {\em a backward orbit} of $x\in X$  if $x_1=x$ 
and 
$B^-=\{x_1,x_2,x_3,\dots\}$ with 
$x_n=F(x_{n+1})$ for all $n\ge 1$. 
We say $F$ satisfies {\em strong Parry topological transitivity} (STT) if there 
exists a dense 
backward orbit for some $x\in X$ . 
Clearly STT implies PTT.  The proof of Theorem~\ref{ttiptt} shows that 
under the same hypotheses,  STT is equivalent to TT. 
Note that Example~\ref{first} does not satisfy STT although it does satisfy PTT.

\subsection{PTT for symbolic dynamical systems} 
Consider
the $1$-sided full shift 
\[
\cD^{\BN}=\{\bd=.d_1d_2d_3\dots:d_j\in\cD\},
\]
where $2\le\#(\cD)\le\infty$,
with left shift map $T(.d_1d_2d_3\dots)=.d_2d_3\dots$. Consider also 
the $2$-sided full shift $\cD^\BZ$ with left shift homeomorphism 
$\widetilde T(\dots d_{-1}.d_0d_1d_2\dots)=\dots d_{-1}d_0.d_1d_2\dots$. We use the product topology in 
each case.
If $\#(\cD)<\infty$, then these are compact metric spaces, homeomorphic to the 
Cantor set,  but in any case, they 
are uncountable,
totally disconnected, Polish spaces.

Call a 
subset $X\subseteq  \cD^{\BN}$ a {\em $1$-sided subshift} if it is 
closed and $T$ invariant: $T(X)\subseteq X$. 
Similarly, call a 
subset $Y\subseteq  \cD^{\BZ}$ a {\em $2$-sided subshift} if it is closed and 
$\widetilde T$ invariant: 
$\widetilde T(Y)=Y$. 
The {\em language} $\cL$ of $X$ (or $\cL$ of  $Y$) is the set of all finite words 
$w=w_0w_1\dots w_{\ell-1}$ (we say $|w|=\ell$) so that there exists $\bd=.d_1d_2d_2\dots\in X$ (or $\be=\dots d_{-1}d_0.d_1d_2\dots\in Y$) and $k\in\BN$  
(or $k\in\BZ$) with 
$w_0w_1\dots w_{\ell-1}=d_kd_{k+1}\dots d_{k+\ell-1}$.
Given a $1$-sided subshift $X$, 
we define its {\em natural extension} $\widetilde X$ to be the two sided subshift 
with the same language.
A sub-basis for the topology on $X$ is given by {\em cylinder sets}, which have the form 
\(
[w]=\{\bd\in X:\bd|_{[1,2,\dots,|w|]}=w\}
\), 
for $w\in\cL$. Similarly, a 
sub-basis for the topology on $Y$ is given by {\em cylinder sets}, which have the form 
\(
[w]=
\{\be\in Y:\be|_{[-\ell,-\ell+1,\dots,\ell-1,\ell]}=w\}
\), 
where $w\in\cL$ and $|w|=2\ell+1$. 
The following is an easy characterization of TT for $T:X\to X$ or
$\widetilde T:Y\to Y$.
\begin{lemma}
The $1$-sided {\rm (}or $2$-sided\,{\rm )} shift, $X$ {\rm (}or $Y$\,{\rm )}, is 
topologically transitive {\rm (}TT or TTT\,{\rm )} 
if and only if its language $\cL$ satisfies
\begin{equation}\label{lang}
\forall\ v,w\in \cL\ \exists\  c\in \cL \text{ so that } vcw\in \cL.
\end{equation}
\end{lemma}

\begin{proof}
Suppose $O^+(\bd)$ is dense in the $1$-sided shift $X$. Given 
$v,w\in\cL$, let $[v]$ and $[w]$ be the corresponding cylinder sets. We have
that there exist $n,m\in\BN_0$ so that  $T^n (\bd)\in [v]$ and $T^m (\be)\in [w]$, 
and by Lemma~\ref{infinite}, we may assume $m\ge n+|v|$. We have 
$\bd_{[n,\dots,n+|v|-1]}=v$ and $\bd_{[m,\dots,m+|w|-1]}=w$. 
Since $m>n+|v|-1$, $\bd_{[n,\dots,m+|w|-1]}=vcw$ for some $c\in \cD^{m-n-|v|
+1}\cap\cL$.
Now suppose $O(\be)$ is dense in  $Y$. We are done if
$O^+(\be)$ is dense, 
so assume that $O^-(\be)$
is dense. Then there are 
$n,m\in\BN$ so that  
$T^{-n} (\be)\in [v]$ and $T^{-m} (\be)\in [w]$, and   
 $-n+|v|\le -m$. Thus 
$\be_{[-n,\dots,-m+|w|-1]}=vcw$
for some $c\in\cL$.

Conversely, suppose $\cL$ satisfies (\ref{lang}).
Enumerate 
$\cL=\{w_1,w_2,w_3\dots\}$, and by induction, choose a sequence 
$c_1,c_2,c_3,\dots\in \cL$ so that 
$w_1c_1w_2c_2\dots w_{n-1}c_{n-1}c_n\in\cL$ for all $n$.
Then 
\(
\bd=w_1c_1w_2c_2w_3c_3w_4\dots\in X
\)
 and $O^+(\bd)$ is dense. 
 In a similar way, if $\cL$ satisfies (\ref{lang}) for a two sided shift $Y$, then there exists 
\(
\be=\dots w_3b_2w_2b_1w_1c_1w_2\dots
\)  
with $O(\be)$ dense. 
\end{proof}

\begin{cor}\label{eight}
The $2$-sided natural extension $\widetilde X$ satisfies TTT 
{\rm (}equivalently, TT{\rm )} if and only if the corresponding $1$-sided shift 
$X$ satisfies TT.
\end{cor}

\begin{cor}\label{nine}
If a $1$-sided shift $X$ satisfies TT then it satisfies PTT.
\end{cor}

\begin{proof}
The natural extension $\widetilde X$ of $X$ is TT, and there 
exists $\tilde\bd\in \widetilde X$ so that $O^-(\tilde\bd)$ is dense. 
The $1$-sided factor map $\pi_+:\widetilde X\to X$, defined 
$\pi_+(\dots d_{-1}d_0.d_1d_2\dots)=.d_1d_2\dots$, is open and 
satisfies $\pi_+(\widetilde T(\bd))=T(\pi_+(\bd))$. Then for $\bd=\pi_+(\tilde\bd)$,
we have that $O^-(\bd)\subseteq\pi_+(O^-(\tilde \bd)$ is dense in $X$.
\end{proof}

\noindent Note that when $\#(\cD)<\infty$, Corollary~\ref{eight} and Corollary~
\ref{nine} follow from 
Theorem~\ref{ttiptt}. However, we will also be interested in the case $\#(\cD)=
\infty$.

\begin{example}\label{second}
Let $X\subseteq \{1,2,\overline 1,\overline 2\}^\BN$ be the subshift defined 
by forbidding the words
\(
\cF=\{\overline k\ell:k,\ell\in\{1,2\}\}.
\)
Here we have two $1$-sided $2$-shifts: 
``unbarred" $\{1,2\}^\BN$ and ``barred'' $\{\overline 1,\overline 2\}^\BN$, 
with the possibility of ``barring'' the tail of a point $\bd\in\{1,2\}^\BN$.
Any point $\bd\in X$ has $T^n(\bd)\in\{1,2\}^\BN\cup\{\overline 1,\overline 2\}^\BN$ for $n$ sufficiently large, so $O^+(\bd)$ is never dense.  
However, any $\overline\bd\in\{\overline 1,\overline 2\}^\BN$
has $O^-(\overline\bd)$ dense. 
\end{example}

\section{Piecewise interval maps}\label{pim}

Let $\lambda$ denote Lebesgue measure on $[0,1)$.
An {\em interval partition} 
is a finite or countable indexed collection 
\(
 \xi=\{\Delta(d)\subseteq[0,1):d\in\cD\}
\)
of $2\le\#(\xi)\le \infty$ disjoint intervals, with $\lambda(D)=1$, where 
\(
D=
\bigcup_{d\in\cD}\Delta(d).
\)
The intervals $\Delta(d)$, which have endpoints $a_d<b_d$, may be open, closed, half open 
$(a_d,b_d]$ or half-closed $[a_d,b_d)$.
Let $\Delta^\circ(d)=(a_d,b_d)$ and note that $\cup_{d\in\cD}\Delta^\circ(d)$ is always dense 
and open. 
We generally refer to elements of the 
index set $\cD$ as {\em digits}. 

A {\em piecewise interval map} (PIM) $F$ on $[0,1)$ is 
an interval partition $\xi$, together with 
a map
$F:D\rightarrow [0,1)$
such that 
\begin{enumerate}
\item each $F|_{\Delta(d)}$ is continuous and strictly monotonic,
\item $\lambda(B^c)=0$  where $B=\{x:F^n x\in D\text{ for all }n\ge 0\}$,
\item for all $d,d'\in\cD$ (including $d=d'$)  and $n\ge 0$, $\Delta(d)\cap T^{n}\Delta(d')$ is 
either an interval or empty (i.e., 
it does not consist of a single point).
Equivalently, $\Delta^\circ(d)\cap T^{n}\Delta^\circ(d')\not=\emptyset$.
\end{enumerate}

We often also assume that $F$ is surjective (and this is clearly necessary for $F$ to satisfy 
TT), although we do not require this.
For a PIM $F$, we say $F|_\Delta$ is  
{\em Type A} if it is increasing and {\em type B} if it  
is decreasing. We say $F$ is type A (or type B) if every $F|_\Delta$ is type A
(or type B). Otherwise, $F$ is called {\em mixed type}.
We say $F$ is {\em full} on $\Delta\in\xi$ if $F(\Delta)=[0,1)$. 
Condition (4) can always be achieved by taking each $\Delta\in\xi$ to be an open interval. The 
process of 
removing some endpoints from $\xi$ to make $F$ satisfy (4) 
only changes $D$ on a countable set. However, in 
certain examples, it is natural to keep the endpoints (see the examples below).

Since each $F|_{\Delta}$ is strictly monotonic, condition (3) is automatic if 
$D^c$ is countable. In particular, (3) always holds if $\xi$ is finite. 
Condition (3) also holds if $\lambda\left(\{x:F'(x)=0\}\right)=0$, since
this is equivalent  to $F$ being {\em nonsingular} in the sense that 
$\lambda(F^{-1}E)=0$ for each $E\subseteq[0,1)$
with $\lambda(E)=0$.   

In many cases (see e.g., \cite{karmacor}, \cite{boyarsky}) one can show more. 
We say a measure $\mu$ on $[0,1)$ is an  {\em absolutely continuous} 
$F$-invariant measure, {\em equivalent} to Lebesgue measure (ACIM),
If there exists an integrable function with $\rho(x)>0$ for $\lambda$ a.e.  
$x\in[0,1)$ so that $\mu$, defined by 
$\mu(E)=\int_E\rho(x)\,dx$, is $F$-invariant, namely, 
$\mu(F^{-1}E)=\mu(E)$. 
In particular, the existence 
of an ACIM implies that $F$ is nonsingular. 
Quite often one can also show that this  ACIM $\mu$ 
is an {\em ergodic} measure for $F$ (see the examples below). 
The following is a routine application of the Birkhoff ergodic theorem.

\begin{prop}\label{ergodic}
If a PIM $F$ has an ergodic ACIM
then 
$F$  satisfies TT. 
\end{prop}

\subsection{$F$-representations}

Recall that $D=\cup_{d\in\cD}\Delta(d)\subseteq[0,1)$.  
By abuse of notation, 
we also denote by $\xi$ the map $\xi:D\to\cD$ 
with $\xi(x)=d$ for $x\in\Delta(d)$.
Given a PIM $F$, we define 
the $F$-representation 
of $x\in B$ to be the the sequence 
\[
\br(x)=.d_1d_2d_3d_4\dots\in\cD^{\BN},
\] 
where $d_n=\xi(F^{n-1}x)$ for $n\in\BN$. 
In ergodic theory, $\br(x)$ is called the $(F,\xi)$-name of $x$. 
We say that $F$-representations are {\em valid} 
if the map $\br$ is  
injective for $\lambda$ a.e. 
$x\in B$.

Parry observed in his paper \cite{parry2} that the previous conditions 
for validity
(i.e., in \cite{bissinger}, \cite{everett}, \cite{renyi})  were sufficient conditions, and were
``metric'' in nature. 
Probably the nicest result of this type 
is Kakeya's Theorem \cite{Kakeya}, which essentially says 
that $F$-representations are valid 
for PIMs $F$ of type A or B, 
provided $|F'(x)|>1$ almost everywhere. 
Parry observed that one ought to expect the necessary and sufficient conditions 
for validity to 
be {\em dynamical} in nature. He went on to prove that 
$F$-representations are valid if $F$ satisfies 
what we have called Parry topological transitivity.

\subsection{Examples}

\begin{example}[$\beta$-representations]
Consider the type A maps $F:[0,1]\to[0,1]$ defined by
$F(x)=\beta x\text{\ mod }1$, for 
$\beta>1$. Here $\xi(x)=\lfloor\beta x\rfloor$ with $\cD=\{0,1,\dots,\beta-1\}$ for 
$\beta\in\BN$, and 
$\cD=\{0,1,\dots,\lfloor\beta\rfloor\}$ for $\beta\not\in\BN$.
The $\beta$-representations were introduced in \cite{renyi}, who 
showed that every $\beta$-transformation $F$ 
has an ergodic ACIM (so satisfy TT). An explicit formula for the density $\rho(x)$ 
was 
given by Parry \cite{parry1}).
Parry (see \cite{parry2}) studied the more general 
$\alpha$-$\beta$-transformation, 
$F(x)=\alpha +\beta x\text{\ mod }1$, which he showed are not necessarily 
ergodic (or topologically transitive).
\end{example}

\begin{example}[Generalized Gauss transformations]
For real numbers $r\ge1$, define the type 2 map 
$F(x)=r/x\text{ mod }1$ 
with $\xi(x)=\lfloor r/x \rfloor$. 
The case $r=1$, known as the {\em Gauss transformation}, has an an ergodic 
ACIM with $\rho(x)=(\log(2)(x+1))^{-1}$. The existence of an ergodic ACIM for 
$r>1$ 
is proved in 
\cite{erblin} (an explicit formula for each $r\in\BN$ is given in  \cite{cor}).
Thus each such $F$ satisfies TT. The corresponding $F$-representations 
are (generalized) continued fraction coefficients. 
\end{example}

\begin{example}
[Quadratic maps]
For $s\approx 0.8$, $s<r\le 1$, consider $F:[0,1]\to[0,1]$ by
\[ 
 F(x)=-4r\left( 
 (1 - r - 4 r^2 + 4 r^3)-(1 - 8 r^2 + 8 r^3) x +
 r(1 - 2 r)^2  x^2)\right),
 \]
with $\xi(x)=0$ if $x<(1+2r-4r^2)/(2r-4r^2)$ and $\xi(x)=1$ otherwise (this is the map 
$q(x)=4rx(1-x)$, restricted to the interval $[q(r),r]$, then renormalized). 
These maps are commonly studied in chaos theory (see \cite{devaney}).
There is a set $r$ of positive Lebesgue measure with an ergodic ACIM and hence 
TT.
It is known that there is a set of values for $r$ of positive Lebesgue measure so 
that 
$F$ has an ergodic ACIM and hence is TT.  
Closely related to both the quadratic maps and $\beta$-transformations 
are the {\em tent maps} defined for $1<\tau\le 2$ by $P(x)=\tau\,x \text{\ wod\ }
1$,
where we define $y \text{\ wod\ }1=y\text{\ mod\ }1$ if $\lfloor y\rfloor$ is even, 
and
$1-(y\text{\ mod\ }1)$ if $\lfloor y\rfloor$ is odd.
For all $\tau$ sufficiently large, $F$ has an ergodic ACIM and hence is TT (see 
\cite{gora}).
\end{example}

\begin{example}[The Cantor map]\label{cantor}
This map $F$ is defined to be linear, increasing, and 
full on each intervals $\xi$ in the complement $K^c$ of the Cantor set
$K$. The intervals
in $\xi$ are naturally indexed by $\cD=\BZ[1/2]\cap (0,1)$, the dyadic rationals in 
$(0,1)$. 
Note that in this example $D=K^c$ is measure zero but uncountable. 
More generally, 
$\xi$ can be replaced 
by any interval partition. Such maps are called {\em generalized L\"uroth 
transformations} in \cite{karmacor} in the case of $\xi$ finite. All such maps are 
TT and ergodic for Lebesgue measure. 
\end{example}

\begin{example}[Generalized Egyptian fracions]\label{ef}
Define 
$F(x)=x-1/\lceil1/x\rceil$ and $\xi(x)=\lceil 1/x\rceil$.
Note that $O^-(0)$ is dense, so $F$ satisfies PTT, whereas 
$F^n(x)\searrow 0$ for all $x$, so 
$F$ does not satisfy TT.  
Note also that $O^{-}(x)$ is dense only for 
$x=0$,  and not for a dense $G_\delta$ set of $x$. 
Here, $B$ is the set of irrationals, and 
$x=1/d_1+1/d_2+1/d_3\dots$ 
is the infinite {\em greedy Egyptian fraction} expansion of an irrational $x$. 
More generally, for a strictly increasing sequence $\ba=(a_1,a_2,a_3,\dots)$ of positive 
integers, $a_1>1$, such that $1\le \sum 1/a_n\le \infty$ 
(e.g. the primes).  Let $\lceil y\rceil_\ba=a_n$ where $a_{n-1}<y\le a_n$, and
$F(x)=x-1/\lceil1/x\rceil_\ba$. The case $a_n=2^n$ gives binary expansions. 
\end{example}

\begin{example}
[Interval exchange transformations]
Let $\xi$ be an interval partition, and let
$\xi'$ be a ``permutation'' of $\xi$. In particular,  
suppose there is a bijection $\varphi:\xi\to\xi'$
such that  each $\Delta\in\xi$ there is $r(\Delta)\in(-1,1)$ so that 
$\varphi(\Delta)=\Delta+r(\Delta)$. 
Define $F(x)=x+r(\Delta)$ for $x\in\Delta$
(see \cite{keane}, \cite{iet}). 
Interval exchanges preserve Lebesgue measure. Various conditions for ergodicity 
and TT are known (see \cite{keane}, \cite{veech}, \cite{masur}). 
Included here 
are the circle rotations $F$, which can be realized as $2$-exchanges 
$\xi=\{[0,\alpha),[\alpha,1)\}$ (labeled $0$ and $1$), with  
TT and ergodicity if and only if  $\alpha\not\in\BQ$.
The resulting $F$-representations are Sturmian sequences.
Similarly, the von Neumann adding machine transformation $F$ 
is an exchange of the partition $\xi$ into intervals of lengths $1/2^n$, in order of 
decreasing length, $\xi'$ the partition into the same intervals, but in order of increasing lengths. 
This is TT and ergodic.
Up to metric isomorphism, any ergodic measure preserving transformation $F$ can be 
realized as a (usually infinite) interval exchange (see \cite{aow}).  It should be noted that 
interval exchange transformations $F$
differ from the other examples discussed here
because they are invertible.     
Orientation reversing interval exchange transformations were studied in \cite{nog},
but they rarely satisfy TT.
\end{example}

\section{Parry's Theorem}

In this section we state and prove our main results about topological transitivity and valid $F$-expansions for piecewise 
interval maps $F$. The first result is Parry's theorem \cite{parry2}.
Our contribution is to extend the proof to the mixed type case. 

\begin{thm}[Parry \cite{parry2}]
\label{pt}
Suppose $F$ is a PIM {\rm (}type A, type B or mixed type{\rm )}. 
If $F$ satisfies PTT, then 
$F$-representations are valid. 
\end{thm}

\noindent Parry also proved the following partial converse, which we prove below for 
convenience. 

\begin{prop}[Parry, \cite{parry2}]\label{pconv}

Let $F$ be a PIM so that $F^{-1}(0)$ includes all the endpoints of $\xi$ except possibly $0$ or 
$1$. If $F$-representations are valid then $F$ satisfies PTT. 

\end{prop}

\noindent We also prove this below. 
Next, we state  our ``modern'' version of Parry's Theorem.

\begin{thm}\label{npt}
Suppose $F$ is a PIM {\rm (}type A, type B or mixed type{\rm )}. If $F$ satisfies TT then 
$F$-representations are valid.
\end{thm}

\subsection{Some preliminaries} 

Let $F$ be a PIM. 
An interval $I\subseteq[0,1)$ is called a {\em homterval}
if $F^n|_{I}$ is continuous and strictly monotonic  for each $n\ge1$. In particular, $F$ is 
a 
homeomorphism between $I$ and each  
$F^n(I)$. There are two special kinds of homtervals. A homterval $I$
is called a  {\em wandering interval} if 
$F^n(I)\cap F^m(I)=\emptyset$ for all $m>n\ge 0$.
A homterval $I$ is called {\em absorbing interval} with period $p\ge 1$ if 
$I, F(I), \dots, F^{p-1}(I)$ pairwise disjoint,
and $F^p(I)\subseteq I$. Here we have that
$F^p_I:I\to J$, for $J=F^p(I)$ is a homeomorphism.
The following is a basic result of 1-dimensional dynamics (see \cite{vanstrien}).

 \begin{lemma} \label{fundd}
 If $J$ is a homterval, then either $J$ is a wandering interval or $J\subseteq I$ for 
 an absorbing interval $I$ with some order $p$.
 \end{lemma}

\begin{proof}
Suppose $J$ is a homterval that is not a wandering interval.
Then  there exist $n\ge 0$ and $p\ge 1$ so that $F^nJ\cap F^{n+p}J\not=\emptyset$,
and the interval $F^nJ\cup F^{n+p}J$ is a homterval.
Repeatedly applying $F^p$ gives 
$F^{n+\ell p}J \cap F^{n+(\ell+1)p}J\not=\emptyset$ for each $\ell\ge 0$.
It follows that $I=\bigcup_{\ell=0}^\infty F^{n+\ell k} J$ is a homterval  with $J\subseteq I$.
Moreover, $I,F(I),\dots,F^{p-1}(I)$ are pairwise disjoint and $F^p|_I$ is a homeomorphism from 
$I$ onto a
subinterval: $F^p(I)\subseteq I$. 
\end{proof}

\begin{lemma}\label{nohom}
If a PIM $F$ satisfies TT, then there can be no homtervals.
\end{lemma}

\begin{proof}
Suppose to the contrary that $J$ is a homterval and $O^+(x)$ is dense. 
If $J$ is a wandering interval, then $O^+(x)$ can meet $J$ at most once. This contradicts 
the density of  $O^+(x)$. 
By Lemma~\ref{fundd}, the only other possibility is that 
there is period $p\ge 1$ absorbing interval $I$ with $J\subseteq I$.
Since 
$O^+(x)$ is dense, $F^n(x)\in O^+(x)\cap I$ for some $n\ge 0$. We can 
assume without loss of generality that $n=1$ so that $x\in I$. 
We claim this implies that $O^+(x)\cap I$ 
is not dense, which is a contradiction.

To prove the claim we note that 
 $F^n(x)\in I$ only if $n=kp$ for some $k$ (since $I$ has period $p$), 
so without loss of generality we may assume $p=1$, and assume  
$F$ maps $I$ homeomorphically onto $F(I)$. 
Let $x\in I$ and consider $O^+(x)$.
One possibility is that $F(x)=x$, in which case $O^+(x)=\{x\}$ which is not dense in 
$I$. Now we divide into two cases: either $F|_I$ is strictly increasing or $F|_I$ is strictly 
decreasing. In the increasing case, suppose $F(x)\not=x$, and assume without loss of generality that $F(x)>x$. Then there is fixed point $F(y)=y\in I$ so that $F^n(x)$ is increasing 
and $F^n(x)\nearrow y$. Again 
$O^+(x)$ is not dense. 
In the decreasing case, we replace $F|_I$ with $(F^2)|_I$, which is 
increasing.

\end{proof}

Note that a period $p$ absorbing interval $I$
always contains at least one point $y$ of period $p$. That point satisfies
$O^+(y)=\{x,F(y),\dots,F^{p-1}(y)\}$, with $F^p(y)=y$. 
The iterates of any non-periodic point $x\in I$ limit onto some 
finite $O^+(y)$ (size $p$), as in the proof.  
This situation is described in \cite{vanstrien} by saying that $F$ has a  {\em periodic attractor}
of period $p$.

\begin{lemma}\label{noia}
If a PIM $F$ satisfies PTT then there can be no absorbing interval.
\end{lemma}

\noindent This observation is essentially due to Parry \cite{parry2}.

\begin{proof}
Let $I$ be an absorbing interval of period $p$. First, as in the proof of Lemma~\ref{nohom}, we assume without loss of generality 
that $p=1$, so that $F|_I:I\to J$, $J=F(I)\subseteq I$, is a homeomorphism. 
If $x\not\in J$ then then $F^{-1}(x)\cap I=\emptyset$, so $O^-(x)$ cannot be dense. 
Thus we assume $x\in J$, and show that $O^-(x)$ is not dense in $J$.

Consider the homeomorphism $(F|_I)^{-1}:J\to I$. 
We assume without loss of generality that $(F|_I)^{-1}$ is increasing (otherwise replace $F|_I$ with $(F|_I)^{2}$ and $J$ with $(F|_I)^{2}(I)$).
If there is an 
$n>0$ so that $(F|_I)^{-n}(x)\not\in J$ then 
$O^-(x)\cap J$ is finite. Thus we 
assume $(F|_I)^{-n}(x)\in J$ for all $n\ge 0$.  One possibility is that 
$(F|_I)(x)=x$, but this implies $O^-(x)$ is not dense in $J$. Thus
assume that $(F|_I)(x)>x$ (the case $(F|_I(x)<x$ is analogous). 
This implies that $(F|_I)^{-n}(x)$ is a bounded increasing sequence
(the graph of $(F|_I)^{-1}$ is above the diagonal on a neighborhood of $x$).
In particular, $(F|_I)^{-n}(x)$ is not dense.
\end{proof}

Next we study iterations of the the partition $\xi$. For 
$d_1d_2\dots d_n\in\cD^n$, let 
\[
\Delta(d_1d_2\dots d_n)=\{x:\br(x)_{[1,\dots,n]}=d_1d_2\dots d_n\}.
\]
Equivalently,
\begin{align}
\Delta(d_1d_2\dots d_n)
&=\Delta(d_1)\cap F^{-1}\Delta(d_2)\cap\dots\cap F^{-n+1}\Delta(d_n)\nonumber\\
&=\Delta(d_1)\cap F^{-1}\Delta(d_2d_3\dots d_n)\label{induct}\\
&=\Delta(d_1d_2\dots d_{n-1})\cap F^{-n+1}\Delta(d_n).\nonumber 
\end{align}
By our assumption (3) on $F$, the set $\Delta(d_1d_2\dots d_n)$ is either empty or a nontrivial interval.
In the latter case, we call it a {\em fundamental interval} of order $n$ 
(or a {\em cylinder}).
Let $\xi^{(n)}$ be the interval partition 
into fundamental intervals of order $n$, and 
define $||\xi^{(n)}||=\sup\{|\Delta|:\Delta\in\xi^{(n)}\}$, where 
$|\Delta|$ denotes the length of $\Delta$. 
It is clear that $\br$ is injective if and only if $||\xi^{(n)}||\rightarrow 0$.
In ergodic theory, one usually writes
\[
\xi^{(n)}=\bigvee_{k=1}^{n}F^{-k+1}\xi.
\]
If $||\xi^{(n)}||\rightarrow 0$ then $\xi$  is called a {\em generating partition} for $F$.

\begin{proof}[Proof of Proposition~\ref{pconv}]
Denote the endpoints of $\xi$ by $|\xi|$. By the hypotheses 
$|\xi|= F^{-1}(0)\cup\{0,1\}$, and similarly 
$|\xi^{(n)}|=\bigcup_{k=0}^{n-1} F^{-k}(0)\cup\{0,1\}$. 
 Since $F$-representations are valid,
  $||\xi^{(n)}||\rightarrow 0$, which implies 
 $O^-(0)\cup\{0,1\}=\bigcup_{n\ge 1}|\xi^{(n)}|$ is dense.
 It follows that $F$ is PTT.
\end{proof}

For $x\in B$, let $\Delta^n(x)$ be the interval 
in $\xi^{(n)}$ that contains $x$. 
Thus, 
 $||\xi^{(n)}||\not\longrightarrow 0$ if and only 
if there exists an $x$ so that $|\Delta^n(x)|\not\longrightarrow 0$.
Note that 
$\Delta^{n+1}(x)\subseteq\Delta^n(x)$. Define
\[
\Delta(x)=\bigcap_{n\in\BN}\Delta^n(x)
\]
Either $\Delta(x)$
 is a (nontrivial) 
interval or $\Delta(x)=\{x\}$, with 
the former if and only if 
$|\Delta^n(x)|\not\longrightarrow 0$, (i.e., 
if and only if
$F$-representations are not valid). 

All $y\in \Delta(x)$ satisfy $\br(y)=\br(x)$ and 
$\Delta(y) = \Delta(x)$.
When $\Delta(x)$ is a nontrivial interval, each map $(F^n)|_{\Delta(x)}$,
for $n\in\BN$, 
is continuous and strictly monotonic  (i.e., a homeomorphism onto its 
range). In particular, such an interval $\Delta^n(x)$ is a homterval. 
We summarize these last few paragraphs in a lemma.

\begin{lemma}\label{summ}
If $F$-representations are not valid then there exists $x\in B$ so that 
$\Delta(x)$ is a homterval.
\end{lemma}

\begin{proof}[Proof of Theorem~{\rm \ref{npt}}]
Suppose $F$-representations are not valid. By Lemma~\ref{summ} there is a homterval $\Delta(x)$, and by Lemma~\ref{nohom}, $F$ cannot be TT.  
\end{proof}

\subsection{Flip lexicographic order}

Let $\cA=\{d\in\cD:F|_{\Delta(d)}\text{ is increasing}\}$ and 
$\cB=\{d\in\cD:F|_{\Delta(d)}\text{ is decreasing}\}$, 
so that  
$\cD=\cA\cupdot\cB$ is a disjoint union.
Note that $\cD=\cA$ if $F$ is type A, and $\cD=\cB$ if $F$ is type B.
For two intervals $\Delta,\Delta'\in \xi$ say $\Delta<\Delta'$ if $x<x'$ for all 
$x\in\Delta$, $x'\in\Delta'$. This induces an order on $\cD$ by $d<d'$ if $\Delta(d)<\Delta(d')$.
This order, in turn, leads to the following order on $\cD^{\BN}$ called {\em flip lexicographic order}. 

\begin{definition}
Suppose $\cD\subseteq\BZ$. Given 
$\bd=.d_1d_2d_3\dots\in\cD^\BN$, $\be=.e_1e_2e_3\dots\in \cD^{\BN}$,
with $\bd\not=\be$,
let $n=\min\{j\ge 1: d_j\not= e_j\}$. Let $p=0$ if $n=1$ and otherwise 
$p=\#\{j=1,\dots,n-1:d_j=e_j\in \cB\}$. 
Define 
$\bd\prec\be$  if 
$d_n<e_n$ and $p$ is even, or if $d_n>e_n$ and $p$ is odd.
Otherwise, define
$\be\prec\bd$. 
We will write 
$\bd\preceq\bd$ if $\bd\prec\be$ or $\bd=\be$.
\end{definition}

If $F$ is type A, this is lexicographic order, and if $F$ is Type B, it is alternating 
lexicographic order. Parry's proof \cite{parry2} of Theorem~\ref{pt}  
 assumes one of these two cases. Flip lexacographic order  
appears in \cite{milnorthurston}.

\begin{lemma}\label{order}
If $x<y$ then $\br(x)\preceq\br(y)$. Conversely, if $\br(x)\prec\br(y)$ then $x<y$. In particular, if 
$\br(x)\not=\br(y)$ then $x\not =y$.
\end{lemma}

\begin{proof}
Let $x<y$ and 
$\bd=\br(x)$ and $\be=\br(y)$.
One possibility is that $y\in \Delta(x)$, so $\Delta(x)=\Delta(y)$, in which case, $\bd=\be$. 
Otherwise there is a smallest $n\ge 1$ so that $\Delta^n(x)\not=\Delta^n(y)$.
If $n=1$, then
$\Delta^1(x)=\Delta(d_1)<\Delta(e_1)=\Delta^1(y)$, so $d_1<e_1$. Since 
$p=0$, this implies $\br(x)\prec \br(y)$. 
If $n>1$ then 
$x,y\in\Delta(d_1d_2\dots d_{n-1})$
and $p\le n-1$.
If $p$ is even, 
$F^{n-1}|_{\Delta(d_1d_2\dots d_{n-1})}$ is increasing, 
and since $x<y$,  $F^{n-1}(x)<F^{n-1}(y)$. 
We then have $\Delta(d_n)<\Delta(e_n)$ so that $d_n<e_n$. This implies $\bd\prec\be$ since 
$p$ is even.
If, on the other hand,   $p$ is odd, then $F^{n-1}|_
{\Delta(d_0d_1\dots d_{n-1})}$ is decreasing, and  
$x<y$ implies $F^{n-}(y)<F^{n-1}(x)$, which implies $\Delta(e_n)<\Delta(d_n)$ and $e_n<d_n$. 
Since $p$ is odd, this still implies $\bd\prec\be$.

Conversely, suppose $\br(x)\prec \br(y)$.  If $d_1<e_1$ then 
$\Delta(d_1)<\Delta(e_1)$ and $x<y$. 
Now suppose 
 $x,y\in{\Delta(d_1d_2\dots d_{n-1})}$, but 
$d_n\not= e_n$. Since $\bx\prec\by$, we have 
$d_n<e_n$ if $p$ is even and $e_n<d_n$ if $p$ is odd.  
In the first case we have $F^n(x)<F^n(y)$ and in the second, 
$F^n(y)<F^n(x)$ (because $F^n(x)\in\Delta(x_n)$, and likewise for $y$).
Note that $F^{n}|_{\Delta(x_0,x_1,\dots x_{n-1})}$ is continuous, 
and either increasing or decreasing, depending on whether $p$ 
is even or odd.
In both cases, this implies $x<y$. 
\end{proof}

\begin{lemma}
Let $F$ satisfy PTT, and let $x$ be such that $O^-(x)$ is dense
in $[0,1)$. Then $\Delta(x)=\{x\}$.
\end{lemma}

\begin{proof}
If $\Delta(x)\not=\{x\}$, then by Lemma~\ref{summ}, $\Delta(x)$ is a homterval. 
Since $F$ satisfies PTT, Lemma~\ref{noia} implies $\Delta(x)$ cannot be 
an absorbing interval, so by Lemma~\ref{fundd}, $\Delta(x)$ must be a wandering interval. We 
show this is impossible. 

Suppose $F^n(\Delta(x))\cap F^m(\Delta(x))=\emptyset$ for all $m>n\ge 0$. 
This is equivalent to 
$F^{-m}(\Delta(x))\cap F^{-n}(\Delta(x))=\emptyset$ for all $n>m\ge 0$. 
Now
$F^{-n}(x)\subseteq F^{-n}(\Delta(x))$ for all $n$, 
but this containment is never dense. It follows that 
$O^-(x)=\cup_{n\ge 0}F^{-n}(x)$ cannot be dense in $[0,1)$.  

Thus $\Delta(x)=\{x\}$ as claimed. 
\end{proof}

\begin{proof}[Proof of Theorem~\ref{pt}] 
First note that $\Delta(z)=\{z\}$ whenever $Fz=y$ and $\Delta(y)=\{y\}$. 
Thus for any $x$ with $O^-(x)$ dense, $z\in O^-(x)$ implies  $\Delta(z)=\{z\}$.

Let $u<v$ and take $y,z\in O^-(x)$ so that $u<y<z<v$.
By Lemma~\ref{order}, $\br(u)\preceq\br(y)\prec\br(z)\preceq\br(v)$, so that 
$\br(u)\prec\br(v)$. Then by Lemma~\ref{order} again,  $\br(u)\not=\br(v)$.
\end{proof}

\section{$f$-expansions and a generalization}\label{fexpg}

Given a PIM $F$, define the {\em $F$-shift}
\[
X=\overline{\{\br(x):x\in B\}}\subseteq\cD^{\BN},
\] 
with the left shift map $T$.
Indeed, this is a 1-sided shift since $T(\br(x))=\br(F(x))$. Let $\widetilde X$, with 
$\widetilde T$, be the 2-sided natural extension of $X$, 
and let $\cL$ be the language common to both shifts.

\begin{lemma}
A word $d_1d_2\dots d_n\in\cL$ if and only if 
$\Delta(d_1d_2\dots d_n)$ is an interval, or equivalently, 
$\Delta^\circ(d_1d_2\dots d_n)\not=\emptyset$
\end{lemma}

\begin{proof}
Note that $w\in\cL$ if and only if 
$w=\br(x)_{[1,2,\dots,n]}=.d_1d_2\dots d_n$
for some $x\in B$. Then by (3) and (4),  $\Delta(d_1,d_2,\dots,d_n)$
is an interval.
Conversely, suppose 
$\Delta(d_1,d_2,\dots,d_n)$
is an interval. Let $x\in B\cap 
\Delta(d_1,d_2,\dots,d_n)$.
Then 
$.d_1d_2\dots d_n=\br(x)_{[1,2,\dots,n]}\in \cL$
since $\br(x)\in X$.

\end{proof}

For $w=d_1d_2\dots d_n\in\cL$, let
${\overline\Delta}(d_1,d_2,\dots,d_{n})=[a_n,b_n]$, so $\Delta^\circ(d_1,d_2,\dots,d_{n})=(a_n,b_n)$.
Note that 
${\overline\Delta}(d_1,d_2,\dots,d_{n})\subseteq{\overline\Delta}(d_1,d_2,\dots,d_{n-1})$. 
Thus if
$F$ representations are valid,  
 $|{\overline\Delta}(d_1,d_2,\dots,d_{n})|\to 0$ as $n\to\infty$
for any $\bd=.d_1d_2d_3\dots\in X$. Then 
$\{x\}=\bigcap_n {\overline\Delta}(d_1,d_2,\dots,d_{n})$
and we define $E(\bd)=x$. 
 If $\bd=\br(x)$ for $x_0\in B$, then $x\in\Delta(d_1,d_2,\dots,d_n)$
for all $n$, so in this case, $E(\br(x))=x$. 
We summarize.

\begin{prop} Suppose $F$-representations are valid. Then for every 
$\bd=.d_1d_2d_3\dots\in X$ there exists a unique $x:=E(\bd)\in [0,1]$ so that 
$\{x\}=\bigcap_n {\overline\Delta}(d_1,d_2,\dots,d_{n})$. In particular,
then $E(\bd)=\lim_n a_n=\lim_n b_n$. 
If $x\in B$ and $\bd=\br(x)$ then  
$E(\bd)=x$.  
\end{prop}

\begin{lemma}\label{inforb}
If $O^+(x)$ is dense and 
$\Delta^\circ(d_1,d_2,\dots,d_n)\not=\emptyset$, then 
$\{N:F^N(x)\in\Delta^\circ(d_1,d_2,\dots,d_n)\}$ is infinite.
\end{lemma}

\begin{proof}
Since $O^+(x)$ is dense and 
$\Delta^\circ(d_1d_2\dots d_n)$ is nonempty and open, 
there exists smallest $k_1\ge 0$ so that $F^{k_1}(x)\in \Delta^\circ(d_1d_2\dots d_n)$. We will show there exists $k_2>k_1$ so that
$F^{k_2}(x)\in \Delta^\circ(d_1d_2\dots d_n)$.

We know that $\br(F^{k_1}(x))_{[1,2,\dots,n]}=.d_1d_2\dots d_n$ and
$\Delta^\circ(d_1d_2\dots d_nd_{n+1}\dots d_m)\subseteq \Delta^m(F^{k1}(x))$ for all $m>n$. Since $O^+(x)$ is dense, 
$F$ satisfies TT, and thus Theorem~\ref{npt} implies 
$F$-representations are valid. This implies that 
$|\Delta^m(F^{k_1}(x))|\to 0$ as $m\to\infty$. It follows that for some 
$m>n$, which we choose as small as possible, 
$\Delta^\circ(d_1d_2\dots d_{m})$ is properly contained in
 $\Delta^\circ(d_1d_2\dots d_n)$., and 
$F^{k_1}(x)\in \Delta^\circ(d_1d_2\dots d_{m})$.
Then there exists $e_m\not=d_m$ so that 
$\Delta^\circ(d_1d_2\dots d_{m-1}e_m)\not=\emptyset$, 
$\Delta^\circ(d_1d_2\dots d_{m-1}e_m)\subseteq\Delta^\circ(d_1d_2\dots d_m)$ and $F^\ell(x)\not\in \Delta^\circ(d_1d_2\dots d_{m-1}e_m)$ for any $\ell=0,1,\dots, k_1$. Then 
there is a $k_2>k_1$ so that $F^{k_2}(x)\in \Delta^\circ(d_1d_2\dots d_{m-1}e_m)\subseteq\Delta^\circ(d_1d_2\dots d_n)$.
\end{proof}

\begin{prop}\label{istt}
If $F$ satisfies TT then so does the corresponding shift $X$, and $\widetilde X$ satisfies TTT.
\end{prop}

\begin{proof}
For $w_1=d_1d_2\dots d_m,w_2=e_1e_2\dots e_k\in\cL$,  
\[
\Delta^\circ(d_1,d_2,\dots, d_{m_1}),  \Delta^\circ(e_1e_2\dots e_{m_2})\not=\emptyset.\] Choose $x\in B$ so that 
$O^+(x)$ is dense. By Lemma~\ref{inforb}
there exist $k_2>k_1+m_1$ so that $F^{k_1}(x)\in
\Delta^\circ(d_1,d_2,\dots, d_{m_1})$ and $F^{k_2}(x)\in
\Delta^\circ(e_1e_2\dots e_{m_2})$.
Then $\br(x)_{[k_1,\dots,k_1+m_1-1]}=w_1$ and 
$\br(x)_{[k_1,\dots,k_2+m_2-1]}=w_2$. Thus 
$w_1uw_2\in \cL$. 
\end{proof} 

Fixing $d\in\cD$, let $\overline\Delta(d)=[a_d,b_d]$ and let 
$\alpha_d=\lim_{x\to a^+_d} F(x)$ and $\beta_d=\lim_{x\to b^-_d} F(x)$.
Define
$f_d:[0,1]\rightarrow[0,1]$ by
\begin{equation}\label{fddef}
f_d(x)=
\begin{cases}
a_d &\text{\ if\ } 0\le x< F(\alpha_d)\\
(F|_{\Delta(d)})^{-1}(x)& \text{\ if\ } F(\alpha_d) \le x< F(\beta_d)\\
\beta_d &\text{\ if\ } F(\beta)\le x<1\\
\end{cases}
\end{equation} 
Each $f_d$ is continuous because 
$F|_{\Delta(d)}:\Delta(d)\rightarrow[0,1)$ is continuous and monotonic.

\begin{lemma} If $d_1d_2\dots d_n\in\cL$ then 
\[\overline\Delta(d_1d_2\dots d_n)=f_{d_1}(f_{d_2}(\dots f_{d_n}([0,1])\dots)).\]
\end{lemma}

\begin{proof}
For $n=1$ we have $f_{d_1}([0,1])=[a_1,b_1]=\overline \Delta(d_1)$.
Suppose 
\[f_{d_2}(f_{d_3}(\dots f_{d_n}([0,1])\dots))=\overline \Delta(d_2d_3\dots d_n)
=[a',b'],\] where $b'>a'$. Note that $a'$ and $b'$ are  
$f_{d_2}(f_{d_3}
(\dots f_{d_n}(0)\dots))$ and 
$f_{d_2}(f_{d_3}
(\dots f_{d_n}(1)\dots))$ (in one order or the other). Then
\[
f_{d_1}(f_{d_2}(\dots f_{d_n}([0,1])\dots))=f_{d_1}(\overline \Delta(d_2d_3\dots d_n))
=f_{d_1}([a',b']).
\]
Now for any interval $[a',b']$, and any $d\in\cD$, (\ref{fddef}) implies that 
$f_d( [a',b'])=F^{-1}( [a',b'])\cap\overline\Delta(d)$.
The result now follows by (\ref{induct}). 
\end{proof}

\begin{thm}\label{tvalid}
Let $F$ be a PIM such that $F$-representations are valid.  
Then for Lebesgue almost every $x\in [0,1)$ (i.e., for $x\in B$)
\begin{equation}\label{valid}
x=E(\bd)=\lim_{n\to\infty} f_{d_0}(f_{d_1}(\dots f_{d_n}(0)\dots))=\lim_{n\to\infty} f_{d_0}(f_{d_1}(\dots f_{d_n}(1)\dots)),
\end{equation}
where $\bd=.d_0d_1d_2\dots=\br(x)$.  
\end{thm}

For $.d_1d_2d_3\dots\in\cD^\BN$ we call the limits of the type (\ref{valid}) 
{\em generalized $f$-expansions}.  The conclusion of Theorem~\ref{valid}
can be expressed by saying if $F$-expansions are valid, then a.e. $f$-expansion 
converges to ``what it should''. 
This occurs 
whenever $F$ satisfies either TT or PTT. 

Traditionally, additional assumptions on $F$  allow the limits in (\ref{valid}) to be 
expressed in a simpler form. These assumptions, which we try and state here 
fairly generally,  involve a more stringent order relations on the digit set $\cD$.
We say $F$ {\em well ordered} if $\cD\subseteq\BZ$ and $\Delta(d)<\Delta(e)$ if and only if 
$d<e$ (one may need to relabel $\cD$ to make this happen). An example of $F$ that is not well 
ordered is the Cantor transformation in Example~\ref{cantor}.
If $F$ is well-ordered, we
define $f:\BR\rightarrow[0,1)$ by $f(x)=f_d(x-d)$ if $x\in[d,d+1)$
for each $d\in\cD$. 
We extend $f$ to a complete the definition 
of $f$ to a function $f:\BR\to [0,1]$
by defining $f(x)=f(a)$ for all $x<a$, where 
$\Delta(d)=[a,b)$ is the left most fundamental interval, and $f(x)=f(b)$ 
if $[a,b)$ is the first fundamental interval smaller than $x$. 
This is most natural if $F$ is either type A or type B, in which case $f$ is 
continuous, and either increasing or decreasing (not necessarily strictly), 
respectively. 

If we restrict the function $f$, as defined above, to the intervals in $\BR$ on 
which it is strictly monotonic, then $f^{-1}$ exists, and we have 
\[
F(x)=f^{-1}(x)\text{\ mod\ }1.
\]
This is a traditional starting point for the theory (see \cite{Kakeya}, \cite{parry2})
Equivalently, 
we can view $f$ as the inverse of the function $F(x)+\xi(x)$.

Given 
$.d_1d_2d_3\dots\in\cD^\BN$ we define the (classical) $f$-expansion by
\[
f(d_1+f(d_2+f(d_3+\dots))).
\]
In particular, we understand this expression this to be the limit 
\[
\lim_{n\to\infty} f(d_1+f(d_2+f(d_3+\dots f(d_n)\dots)))
\]

\begin{thm}
Suppose $F$ is a well ordered PIM such that $F$-representations are valid (i.e., 
if $F$ satisfies either TT or PTT). Then $f$-expansions are {\em valid} in the sense that 
for 
$\lambda$ a.e $x\in [0,1)$ (i.e., for $x\in B$), $\br(x)=.d_1d_2d_3\dots\in\cD^\BN$ and 
\[
x=f(d_1+f(d_2+f(d_3+\dots))).
\]
We also have $\lim_{n\to\infty} f(d_1+f(d_2+f(d_3+\dots f(d_n+1)\dots)))$
\end{thm}

\section{Topological transitivity implies Parry topological transitivity}

We can now prove our main result.

\begin{thm}\label{mainresult}
If $F$ is a piecewise interval map (PIM) that that satisfies TT, 
then it satisfies PTT. 
\end{thm}

\begin{proof}
Since $F$ satisfies TT,  Proposition~\ref{istt} implies that the 
$2$-sided $F$-shift $\widetilde X$ satisfies TTT. 
Let $\tilde\bd\in \widetilde X$ be such that $O^-(\tilde\bd)$ is 
dense. Let $\bd_n={\widetilde T}^{-n}(\tilde\bd)$, $n\ge 0$, and for each $n$ 
let $\bd_n=\pi_+(\tilde\bd_n)$, where $\pi_+:\widetilde X\to X$, defined
$\pi_+(\dots d_{-1}d_0.d_1d_2\dots)=.d_1d_2\dots$, is the factor map from the 
2-sided to 1-sided shift. Note that 
$\pi_+(\widetilde T(\tilde\bd))=T(\pi_+(\tilde\bd))$, so 
we have 
\[
T^n(\bd_n)=T^n(\pi_+(\tilde \bd_n))=\pi_+(\widetilde T^n(\tilde\bd_n))=\pi_+(\tilde \bd)=\bd.
\]
Let $x_n=E(\bd_n)$, which exists by Theorem~\ref{npt} and Theorem~\ref{tvalid}.
It follows that $F^n(x_n)=x$ so $B=\{x_0,x_1,x_2,\dots\}$ is a backward orbit for $F$ 
and it suffices to show $B$ is dense. But 
$(\widetilde T^{-n}(\tilde \bd))|_{[1,2,\dots,m]}=d_1d_2\dots d_m$
implies $x_n\in\overline\Delta(d_1d_2\dots d_m)$.  
Since $O^-(\tilde\bd)$ is 
dense, $B$ is dense too, and so $F$ satisfies PTT. 
\end{proof}

\bibliographystyle{plain}

\begin{thebibliography}{10}

\bibitem{akin}
Ethan Akin and Jeffrey~D. Carlson.
\newblock Conceptions of topological transitivity.
\newblock {\em Topology Appl.}, 159(12):2815--2830, 2012.

\bibitem{aow}
Pierre Arnoux, Donald~S. Ornstein, and Benjamin Weiss.
\newblock Cutting and stacking, interval exchanges and geometric models.
\newblock {\em Israel J. Math.}, 50(1-2):160--168, 1985.

\bibitem{bissinger}
B.~H. Bissinger.
\newblock A generalization of continued fractions.
\newblock {\em Bull. Amer. Math. Soc.}, 50:868--876, 1944.

\bibitem{blanchard}
F.~Blanchard.
\newblock {$\beta$}-expansions and symbolic dynamics.
\newblock {\em Theoret. Comput. Sci.}, 65:131--141, 1989.

\bibitem{boyarsky}
Abraham Boyarsky and Pawe{\l} G{\'o}ra.
\newblock {\em Laws of chaos}.
\newblock Probability and its Applications. Birkh\"auser Boston Inc., Boston,
  MA, 1997.
\newblock Invariant measures and dynamical systems in one dimension.

\bibitem{karmacor}
Karma Dajani and Cor Kraaikamp.
\newblock {\em Ergodic theory of numbers}, volume~29 of {\em Carus Mathematical
  Monographs}.
\newblock Mathematical Association of America, Washington, DC, 2002.

\bibitem{cor}
Karma Dajani, Cor Kraaikamp, and Niels van~der Wekken.
\newblock Ergodicity of {$N$}-continued fraction expansions.
\newblock {\em J. Number Theory}, 133:3183--3204, 2013.

\bibitem{devaney}
Robert~L. Devaney.
\newblock {\em An introduction to chaotic dynamical systems}.
\newblock Studies in Nonlinearity. Westview Press, Boulder, CO, 2003.

\bibitem{everett}
C.~J. Everett.
\newblock Representations for real numbers.
\newblock {\em Bull. Amer. Math. Soc.}, 52:861--869, 1946.

\bibitem{gora}
Pawe{\l} G{\'o}ra.
\newblock Invariant densities for piecewise linear maps of the unit interval.
\newblock {\em Ergodic Theory Dynam. Systems}, 29:1549--1583, 2009.

\bibitem{Kakeya}
S.~Kakeya.
\newblock On the generalized scale of notation.
\newblock {\em Japan J. Math.}, 1:95--108, 1924.

\bibitem{keane}
Michael Keane.
\newblock Interval exchange transformations.
\newblock {\em Math. Z.}, 141:25--31, 1975.

\bibitem{kolyada}
Sergi{\u\i} Kolyada and {{L}}ubom{\'{\i}}r Snoha.
\newblock Some aspects of topological transitivity---a survey.
\newblock In {\em Iteration theory ({ECIT} 94) ({O}pava)}, volume 334 of {\em
  Grazer Math. Ber.}, pages 3--35. Karl-Franzens-Univ. Graz, Graz, 1997.

\bibitem{masur}
Howard Masur.
\newblock Interval exchange transformations and measured foliations.
\newblock {\em Ann. of Math. (2)}, 115(1):169--200, 1982.

\bibitem{erblin}
Erblin Mehmetaj.
\newblock {\em Properties of $r$-continued fractions}.
\newblock PhD thesis, George Washington University, 2014.

\bibitem{milnorthurston}
John Milnor and William Thurston.
\newblock On iterated maps of the interval.
\newblock In {\em Dynamical systems ({C}ollege {P}ark, {MD}, 1986--87)}, volume
  1342 of {\em Lecture Notes in Math.}, pages 465--563. Springer, Berlin, 1988.

\bibitem{anima}
Anima Nagar, V.~Kannan, and S.~P. Sesha~Sai.
\newblock Properties of topologically transitive maps on the real line.
\newblock {\em Real Anal. Exchange}, 27:325--334, 2001/02.

\bibitem{nog}
Arnaldo Nogueira.
\newblock Almost all interval exchange transformations with flips are
  nonergodic.
\newblock {\em Ergodic Theory Dynam. Systems}, 9(3):515--525, 1989.

\bibitem{iet}
V.~I. Oseledec.
\newblock The spectrum of ergodic automorphisms.
\newblock {\em Dokl. Akad. Nauk SSSR}, 168:1009--1011, 1966.

\bibitem{parry1}
W.~Parry.
\newblock On the {$\beta $}-expansions of real numbers.
\newblock {\em Acta Math. Acad. Sci. Hungar.}, 11:401--416, 1960.

\bibitem{parry2}
W.~Parry.
\newblock Representations for real numbers.
\newblock {\em Acta Math. Acad. Sci. Hungar.}, 15:95--105, 1964.

\bibitem{renyi}
A.~R{\'e}nyi.
\newblock Representations for real numbers and their ergodic properties.
\newblock {\em Acta Math. Acad. Sci. Hungar}, 8:477--493, 1957.

\bibitem{vanstrien}
Sebastian van Strien.
\newblock Smooth dynamics on the interval (with an emphasis on quadratic-like
  maps).
\newblock In {\em New directions in dynamical systems}, volume 127 of {\em
  London Math. Soc. Lecture Note Ser.}, pages 57--119. Cambridge Univ. Press,
  Cambridge, 1988.

\bibitem{veech}
William~A. Veech.
\newblock Gauss measures for transformations on the space of interval exchange
  maps.
\newblock {\em Ann. of Math. (2)}, 115:201--242, 1982.

\end{thebibliography}

\end{document}